\documentclass{amsart}
\usepackage{amsmath}
\usepackage{amssymb}
\usepackage{amsthm}
\usepackage{amsfonts}
\usepackage{mathtools}
\usepackage{url}
\usepackage[pdfencoding=auto]{hyperref}
\usepackage{appendix}
\usepackage{verbatim}
\usepackage{graphicx,color}
\allowdisplaybreaks

\usepackage{tikz-cd}

\title{Infinitely many knots with non-integral trace}

\author{A. W. Reid}
\author{N. Rouse}

\address{\newline Department of Mathematics,
\newline Rice University, 
\newline Houston, TX 77005, USA}
\email{alan.reid@rice.edu}
\email{nicholas.rouse@rice.edu}

\thanks{First author supported by NSF grant DMS $1812397$} 

\def\-{\overline}

\def\G{\Gamma}

 \def\H{\mathbb{H}}
 \def\Z{\mathbb{Z}}
 \def\R{\mathbb{R}}
 \def\Q{\mathbb{Q}} 
  
 \def\C{\mathbb{C}}

\def\qed{ $\sqcup\!\!\!\!\sqcap$}

\def\tr{\mbox{\rm{tr}}\, }

\def\G{\Gamma}

\def\<{\langle}
\def\>{\rangle}

\DeclareMathOperator{\SL}{SL} 
\DeclareMathOperator{\GL}{GL}

\newtheorem{theorem}{Theorem}[section]
\newtheorem{lemma}[theorem]{Lemma}

\newtheorem{proposition}[theorem]{Proposition}

\theoremstyle{definition}

\graphicspath{ {./figures/} }

\begin{document}

\begin{abstract}  
We prove that there are infinitely many non-homeomorphic hyperbolic knot complements $S^3\setminus K_i = \H^3/\G_i$ for which $\G_i$ contains elements whose trace is an algebraic non-integer.
\end{abstract}

\maketitle

%
%
\section{Introduction}\label{intro} 
A basic consequence of Mostow-Prasad Rigidity is that if $M=\H^3/\G$ is an orientable hyperbolic 3-manifold of finite volume, then the traces of the elements in $\G$ are algebraic numbers (see \cite[Theorem 3.1.2]{MR}). In addition,
if there is an element $\gamma\in\G$ for which the trace is an algebraic {\em non-integer}, then Bass's Theorem \cite{Bass} implies that $M$ contains a closed embedded essential surface.  The main result of this
note is the following for which we introduce some notation.  Let $K$ be a hyperbolic knot (or link) 
such that $S^3\setminus K= \H^3/\G$, say that $K$ has {\em non-integral trace} (resp. {\em integral trace}) if $\G$ contains an element whose trace is an algebraic non-integer (resp. there is no such element). 

\begin{theorem}
\label{main}
There are infinitely many distinct knots with non-integral trace.\end{theorem}

\noindent Some examples of such knots were already known from the knot tables \cite{Rolf} (e.g. $10_{98}$ and $10_{99}$ as can be readily checked with SnapPy \cite{CDW} or Snap \cite{Snap}) and in \S \ref{knot12},  we include a list of knots up through $12$ crossings that we are able to confirm have non-integral trace.

In the case of links,  it was shown in \cite{ChD} that there exist infinitely many $2$ component hyperbolic links that have non-integral trace.
\\[\baselineskip]
\noindent{\bf Acknowledgements:}~{\em We are very grateful to Shelly Harvey for pointing out the reference \cite{Mur} to us. We are also very grateful to Ken Baker, Neil Hoffman and Josh Howie for comments on an earlier version of this paper that led to the revised \S \ref{questions} below. We would particularly like to thank Howie for allowing us to include his proof that the knot complements constructed in the proof of Theorem \ref{main} contain a closed embedded essential surface that carries an essential simple closed curve isotopic to a meridian, as well as for the tangle decompositions shown in Figure \ref{tangleDecomp}.}

\section{The basic construction}
\label{basic}

Our basic construction is easy to explain. First, for convenience, we recall the following interpretation of the linking number (see \cite[p. 132]{Rolf}). Let $L=J\cup K \subset S^3$ be a $2$-component link,
and let $[\gamma]$ denote a generator of $H_1(S^3\setminus J, \Z)\cong \Z$. The homology class $[K] \in H_1(S^3\setminus J, \Z)$ is represented by $n.[\gamma]$ for some $n\in \Z$, 
and the {\em linking number} of $J$ and $K$ is $n$.

Now let $L=J\cup K \subset S^3$ be a $2$-component hyperbolic link with $S^3\setminus L\cong \H^3/\G$, where $J$ is the unknot and for which the linking number between $J$ and $K$ is $2$ (after a choice of orientation of $J$ and $K$). Now cyclic branched covers of $S^3$ branched over $J$ are all homeomorphic to $S^3$. Moreover, using the definition of the linking number given above, we see that for $d$ odd, the preimage of $K$ in the 
$d$-fold cyclic branched cover is connected. That is to say, 
such $d$-fold cyclic branched covers of $S^3$ branched over $J$ will be knot complements in $S^3$. For $d$ large enough the knots will be hyperbolic as can be seen from Thurston's Dehn Surgery Theorem using the description of these branched covers as orbifold $(d,0)$-Dehn filling on $J$, and subsequent passage to the appropriate $d$-fold cyclic cover of the orbifold.

We will also insist that there exists $\alpha\in\G$ whose trace is an algebraic non-integer. As noted above, it follows that $S^3\setminus L$ contains a closed embedded essential surface.
That the knot complements constructed in the previous paragraph also contain a closed embedded essential surface follows from \cite{GL}, however, it is more subtle to prove that the knots have non-integral trace.  To do this, we need to
analyze the behavior of $\chi_\rho(\alpha)$ on the canonical component of $L$.  In particular, by understanding how $\chi_\rho(\alpha)$ varies on a particular subvariety of the
canonical component of $L$ we will prove that at those characters $\chi_d$ corresponding to $(d,0)$-Dehn filling on $J$ (and where the cusp corresponding to $K$ remains a cusp), 
$\chi_d(\alpha)$ remains an algebraic non-integer.  As is well-known, since non-integral trace is preserved by passage to finite index subgroups (see for example \cite[Corollary 3.1.4]{MR}), it follows that the knots constructed in the previous paragraph have non-integral trace.

\section{Details about \texorpdfstring{$L$}{L}}
\label{details}

The link $L$ we use is {\tt{L11n106}} from Thistlethwaite's table of 2 component links through $11$ crossings \cite{TT}, and shown in Figure 1. As in \S \ref{basic},  $J$ will denote the unknotted
component of $L$, and $K$ the knotted component of $L$, which in this case is the knot $7_6$
of the tables of \cite{Rolf}. The volume of $S^3\setminus L$ is approximately $10.666979133796239$.  We note that several examples were tested before the plan outlined in \S \ref{basic} was pushed through to completion (see
\S \ref{remsnonint} for a discussion of one example that failed).

\begin{figure}[h]
\label{LinkDiagram}
\includegraphics[scale=0.5]{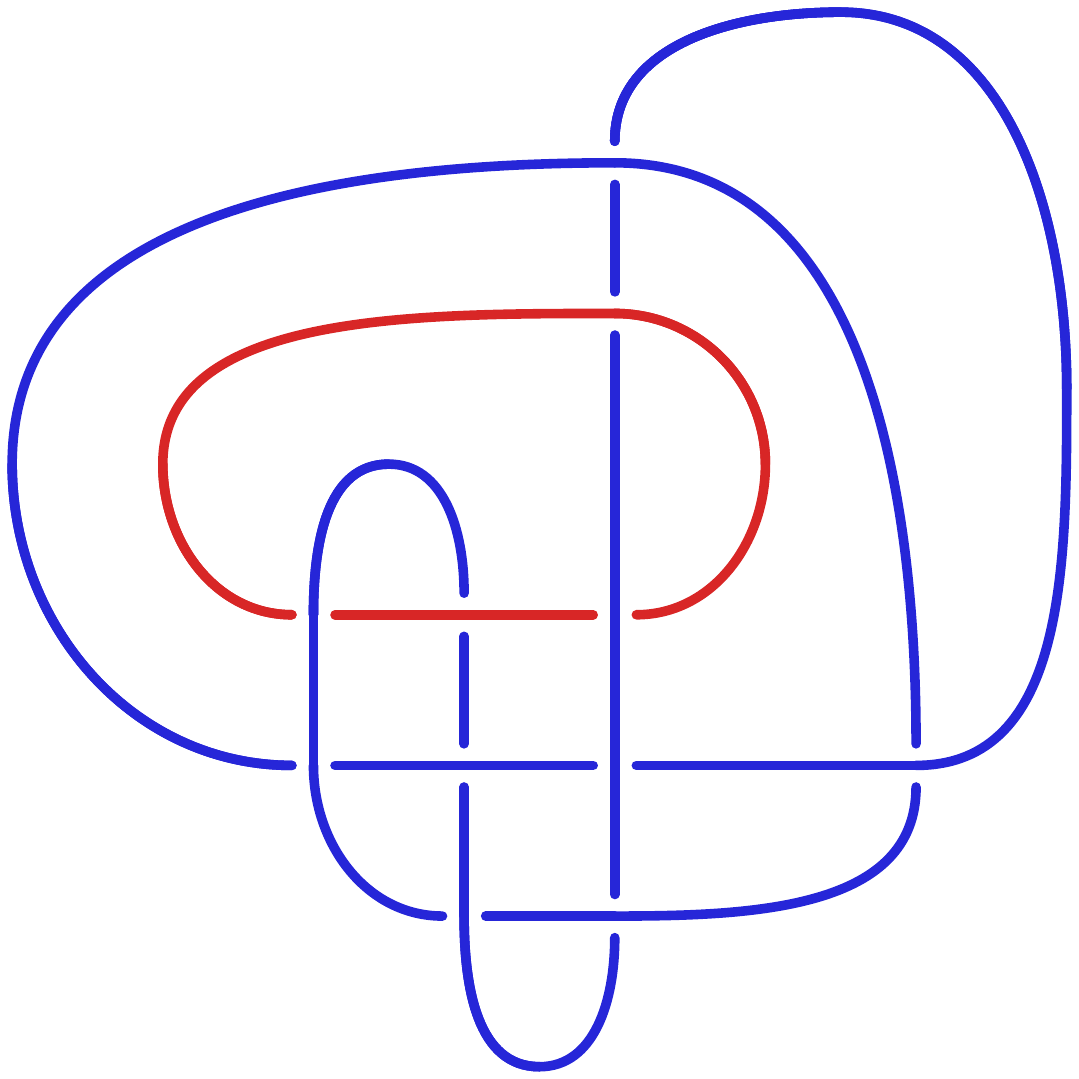}
\caption{The link L11n106. Diagram produced in SnapPy.}
\end{figure}

In the subsections below we gather the details about $S^3\setminus L = \H^3/\G$ that will be used, together with analysis of characters.   We made heavy use of Snap \cite{Snap}, SnapPy \cite{CDW} and Mathematica \cite{Math} in our calculations. 

\subsection{Presentation for \texorpdfstring{$\G$}{Gamma}} 
\label{pres}
\noindent From SnapPy a presentation for $\G$ is given as follows.
\begin{verbatim}
<a,b | abbbaBAbaabABaBAbaabABabbbaBAbaabABBBAbaBAABabAbaBAABabABBBAbaBAAB=1>
\end{verbatim}
\noindent where $A$ and $B$ denote the inverses of $a$ and $b$ respectively. Also from SnapPy meridians for $J$ and $K$ are given by
\begin{verbatim}
J: baabABabbbaBAABabABBBAbaBAABabbbaBA
K: ba
\end{verbatim}
 
\noindent This can be checked by performing $(1,0)$-Dehn filling on $J$ which SnapPy shows results in a manifold homeomorphic to the complement of the knot $7_6$.

Using SnapPy (or Snap) it can be checked that the trace-field of $\Gamma$ is $\Q(\sqrt{-7})$ and that $\tr(a) = \pm (13 + 7\sqrt{-7})/8 $ and $\tr(b)= \pm(17+3\sqrt{-7})/8$ 
and so 
both are algebraic non-integers (this can also be checked using the character variety calculations below).

\subsection{Character variety calculations}
\label{character}
Since $\G$ is $2$-generator, we can conjugate any irreducible representation $\rho:\G\rightarrow \SL(2,\C)$ so that $\rho(a)$ fixes $\infty$ and $\rho(b)$ fixes $0$.  Since we are interested in those
representations $\rho$ for which the meridian of $K$ (identified as $ba$ in \S \ref{pres}) continues to be parabolic, we can normalize so that $\chi_\rho(ba)=-2$ (where the minus sign is chosen so as to be consistent with the output produced by SnapPy). With this arrangement we have:

$$a\mapsto \left(
\begin{array}{cc}
 x & 1 \\
 0 & \frac{1}{x} \\
\end{array}
\right)~\hbox{and}~
b\mapsto \left(
\begin{array}{cc}
 y & 0 \\
 -x y-2-\frac{1}{x y} & \frac{1}{y} \\
\end{array}
\right)$$

\noindent To handle evaluation in Mathematica of the relation on the matrices, we split it up as follows:

\begin{verbatim}
w1 = a.b.b.b.a.B.A.b.a.a.b.A.B;
w2 = a.B.A.b.a.a.b.A.B.a.b.b.b.a.B.A.b;
w3 = a.a.b.A.B.B.B.A.b.a.B.A.A.B.a.b.A.b.a;
w4 = B.A.A.B.a.b.A.B.B.B.A.b.a.B.A.A.B;
\end{verbatim}

\noindent and evaluate

\begin{verbatim}
rel=Factor[w1.w2-Inverse[w3.w4]]
\end{verbatim}

\medskip 

\noindent Setting $X=\chi_\rho(a)$ and $Y=\chi_\rho(b)$ we find that $X$ and $Y$ satisfies $P(X,Y)=0$ where:

\begin{equation*}
\begin{aligned}
P(X,Y)&=X^8 Y+7 X^7 Y^2-2 X^7+21 X^6 Y^3-7 X^6 Y+35 X^5 Y^4-3 X^5 Y^2-8 X^5+35 X^4 Y^5 \\
&\quad+20 X^4Y^3-29 X^4 Y+21 X^3 Y^6+40 X^3 Y^4-39 X^3 Y^2-7 X^3+7 X^2 Y^7+33 X^2 Y^5 \\
&\quad-23 X^2 Y^3-17X^2 Y+X Y^8+13 X Y^6-5 X Y^4-14 X Y^2+X+2 Y^7-4 Y^3
\end{aligned}
\end{equation*}

It is easy to check using Mathematica that $P(X,Y)$ is irreducible over $\Q$, and using the feature {\tt{Factor[*, Extension -> All]}}, Mathematica can check that this is irreducible over $\C$. Indeed, our computations show
that there are two subvarieties in the $\SL(2,\C)$-character variety of $L$, where $ba$ is kept parabolic, and the one above was identified by using the traces of $a$ and $b$ given at the faithful discrete representation. 

Set $t=\chi_\rho(m_0)$ where $m_0$ is the meridian of $J$ described above.  This results in a polynomial $Q(t,X,Y)$ displayed in \S \ref{additional}, and taking the resultant of $P(X,Y)$ and $Q(t,X,Y)$  to eliminate $X$, 
yields the polynomial
$R(t,Y)$ displayed in \S \ref{additional} with highest degree term being $16tY^{24}$.  Thus, if at algebraic integer 
specializations of $t$, the polynomial $R(t,Y)$ remains irreducible, then $Y$ is an algebraic non-integer.
Note that $R(-2,Y)$ is reducible, factoring as
\begin{equation*}
\begin{aligned}
R(-2,Y) &=\left(Y^9+15 Y^8+104 Y^7+435 Y^6+1205 Y^5+2285 Y^4+2956 Y^3+2506 Y^2+1257Y+283\right)^2\\
&\quad \left(2 Y^2-5 Y+4\right) \left(4 Y^2-17 Y+22\right) \left(4 Y^2-11 Y+8\right)
\end{aligned}
\end{equation*}
with the factor corresponding to the complete structure being $4 Y^2-17 Y+22$.

The proof of Theorem \ref{main} will be completed by the following proposition, the proof of which is given in \S \ref{irred}.  For $d$ odd, perform $(d,0)$-Dehn filling on $J$, which amounts to setting $t=2\cos(2\pi/d)$ in $R(t,Y)$.
Now for $d$ odd, $2\cos(2\pi/d)$ is a unit.  To see this, let $\Phi_d(x)$ denote the $d$-th cyclotomic polynomial, and let $\zeta_d$ be a primitive $d$th root of unity. Then $2\cos2\pi/d= \zeta_d+1/\zeta_d$ is a unit if and only if
$\zeta_d^2+1$ is a unit.  By \cite[Lemma 2.5]{Len} this holds if and only $\Phi_d(i)$ is a unit, and this can be deduced from \cite[Lemma 23]{BHM} for example.

\begin{proposition}
\label{specialize}
For infinitely many odd $d>1$, the polynomial $R(2\cos(2\pi/d),Y)$ is irreducible over $\Q(\cos(2\pi/d))$.\end{proposition}

\section{Proving irreducibility}\label{irred}
Our goal in this section is to prove Proposition \ref{specialize}. Rather than working with the polynomial $R(t, Y)$ directly, we will instead consider the polynomial $S(X, Y) = X^8 R(X + X^{-1}, Y)$ (see \S \ref{additional} for an explicit description
of $S$). The reason for making this transformation is the following.
Let $\zeta_d = \exp(2 \pi i /d)$, and note that $S(\zeta_d, Y) = \zeta_d^8 R(2 \cos(2 \pi /d), Y)$, so $S(\zeta_d, Y)$ is irreducible in $\Q(\zeta_d)[Y]$ if and only if $R(2 \cos(2 \pi /d), Y)$ is.  That 
$S(\zeta_d, Y)$ is irreducible in $\Q(\zeta_d)[Y]$ will be established using the following result.

\begin{theorem}\cite[Corollary 1(a)]{DZannier} \label{DZannier}
	Let $k$ be a number field and $k^c$ the field obtained by adjoining all roots of unity to $k$. If $f \in k^c[X,Y]$ and $f(X^m, Y)$ is irreducible in $k^c[X,Y]$ for all positive integers $m \leq deg_Y f$, then $f(\zeta, Y)$ is irreducible in $k^c[Y]$ for all but finitely many roots of unity $\zeta$.
\end{theorem}

\noindent Thus Proposition \ref{specialize} follows immediately from Theorem \ref{DZannier} once we show that $S(X^m, Y)$ is irreducible over $\Q^c = \Q^{ab}$ for each positive integer $m \leq 24$. In fact we will prove that $S(X^m, Y)$ is irreducible over $\overline{\Q}$ for such $m$; that is, $S(X^m, Y)$ is absolutely irreducible. 

To accomplish this, we use \cite{BCG}. Before stating the result of \cite{BCG} that we need, we recall the definition of the Newton polygon of a $2$-variable polynomial.  To that end, let $k\subset \C$ be a field and
$f(X,Y) = \sum_{i,j} c_{i,j} X^iY^j \in k[X,Y]$. The Newton polygon of $P$ is the convex hull in $\R^2$ of all points $(i,j)$ such that $c_{i,j} \neq 0$. We call a point in the Newton polygon a {\em vertex} if it does not belong to the interior of any line segment in the Newton polygon. With this we have the following test for irreducibility.

\begin{theorem}\cite[Proposition 3] {BCG} \label{absoluteirredubility}
Let $k$ be a field and $f(X,Y) \in k[X,Y]$ be an irreducible polynomial. Let $\{(i_1, j_1), \dots, (i_l, j_l)\} \subset \Z^2$ be the vertex set of its Newton 
polygon. If $\mathrm{gcd}(i_1, j_1, \dots,i_l, j_l) = 1$, then $f(X,Y)$ is irreducible over $\overline{k}$.
\end{theorem}
\noindent In our context,  two things need to be established for each $m \leq 24$:

\medskip
\noindent {\bf 1.}~{\em $S(X^m, Y)$ is irreducible over $\Q$;}

\medskip

\noindent {\bf 2.}~{\em the Newton polygon of $S(X^m, Y)$ satisfies the conditions of Theorem \ref{absoluteirredubility}.}\\[\baselineskip]
\noindent {\em Proof of} {\bf 1.} It can be checked quite quickly using Mathematica (for example), that $S(X^m, Y)$ is irreducible over $\Q$ for $m \leq 24$. However, we also 
supply explicit ideals $I_m$ of $\Z[X,Y]$ such that the reduction of $S(X^m, Y)$ modulo $I_m$ is an irreducible polynomial over a finite field. We include below a table of prime numbers $p$ such that $S(X^m, Y)$ is irreducible modulo $I_m = (X-2, p)$.
\begin{center}
\begin{tabular}{c|c c@{\hspace{0.5in}}c|c c@{\hspace{0.5in}} c|c}
    $m$ & $p$ & & $m$ & $p$ & & $m$ & $p$ \\
    \cline{1-2} \cline{4-5} \cline{7-8}
    1 &  17 & & 9 & 17 & & 17 & 11 \\
    2 &  11 & & 10 & 89 & & 18 & 11 \\
    3 & 11 & & 11 & 17 & & 19 & 17 \\
    4 & 31 & & 12 & 11 & & 20 & 53 \\
    5 & 17 & & 13 & 11 & & 21 & 17 \\
    6 & 31 & & 14 & 31 & & 22 & 11 \\
    7 & 11 & & 15 & 17 & & 23 & 11 \\
    8 & 11 & & 16 & 31 & & 24 & 31
\end{tabular}
\end{center}

\noindent {\em Proof of} {\bf 2.} Let us first observe that the effect of replacing $X$ with $X^m$ is to {\em stretch} the Newton polygon of $S(X,Y)$ in the positive $X$-direction.  More precisely, if
$(i,j)$ is a point in the Newton polygon (not necessarily on the boundary) of $S(X,Y)$, then $(mi,j)$ is a point in the Newton polygon of $S(X^m,Y)$. 

From \S \ref{additional}, we observe that $S(X, Y)$ has a $Y$ monomial term with coefficient $1$. Moreover, further inspection of $S(X,Y)$ shows that $1$ is the only power $k$ such that $Y^k$ has nonzero coefficient in $S(X, Y)$. Hence $(0,1)$ is a vertex of the Newton polygon; see Figure \ref{newtonpolygon}.
\begin{figure}[ht]
    \centering
    \includegraphics[scale=0.4]{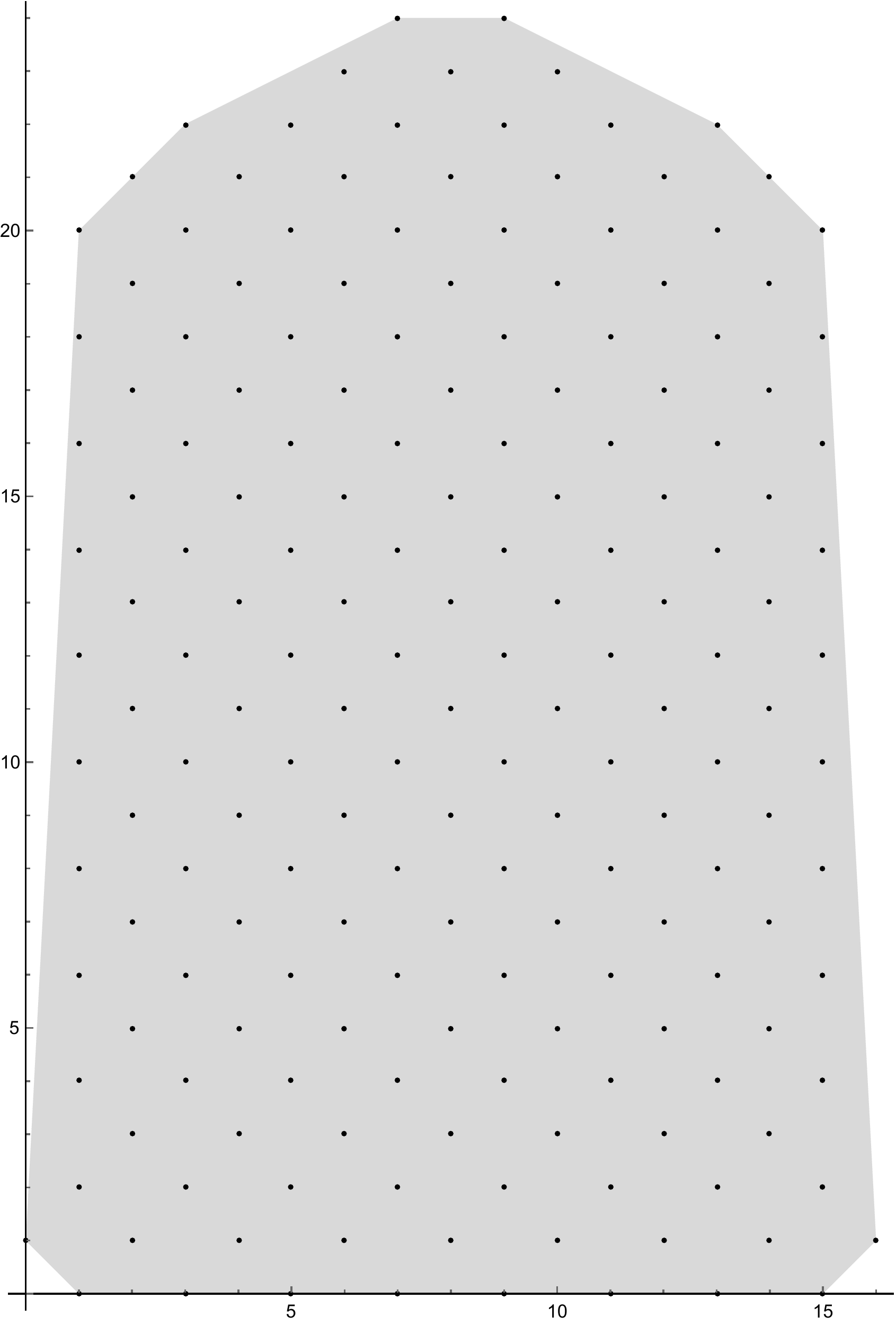}
    \caption{The Newton polygon of $S(X,Y)$ is the convex hull of the black points corresponding to nonzero coefficients of $S(X,Y)$.}
    \label{newtonpolygon}
\end{figure}

In fact $(0,1)$ will be a vertex of the Newton polygon of $S(X^m, Y)$ for all $m \geq 1$. To see this note that after replacing $X$ with $X^m$, the monomial term in $Y$ still has a coefficient $1$ and will remain
the only power of $Y$ that has a nonzero coefficient; i.e.  $(0,1)$ will continue to be a vertex of the Newton polygon of $S(X^m, Y)$ for all $m \geq 1$.  
Thus $S(X^m,Y)$ satisfies the hypotheses of Theorem \ref{absoluteirredubility} whenever $S(X^m, Y)$ is irreducible over $\Q$. \\[\baselineskip]
With these two statements in hand,  we may then conclude that $S(X^m, Y)$ is absolutely irreducible for $m \leq 24$ and hence, by Theorem \ref{DZannier} that $S(\zeta, Y)$ is irreducible over $\Q(\zeta)$ for all but finitely many roots of unity $\zeta$. This, together with the discussion at the start of this section completes the proof of Proposition \ref{specialize}.\qed

\section{Remarks on non-integral trace}
\label{remsnonint}

In this section we gather together some comments about non-integral trace, how it persists in certain Dehn fillings and disappears in others.  
In particular, the example of the link  given in \S \ref{L11n71} stands in contrast to the link $L$ we use in the proof of Theorem \ref{main}, in that, as described in \S \ref{L11n71}, non-integrality does not persist in $(d,0)$-Dehn filling in this case. This clearly needs to be better understood.

\subsection{Some remarks on \texorpdfstring{$S^3\setminus L$}{S\unichar{"00B3} \unichar{"005C} L}} 
\label{frems}

One closed embedded essential surface in the complement of the link $L$ can be constructed from the essential tangle decomposition shown in Figure \ref{4PuncSphere}. The $4$-punctured sphere $S$ shown in Figure \ref{4PuncSphere} is incompressible, and tubing $S$ provides a closed embedded essential surface $F$. \par

\begin{figure}[htbp]
    \centering
    \includegraphics[scale=0.5]{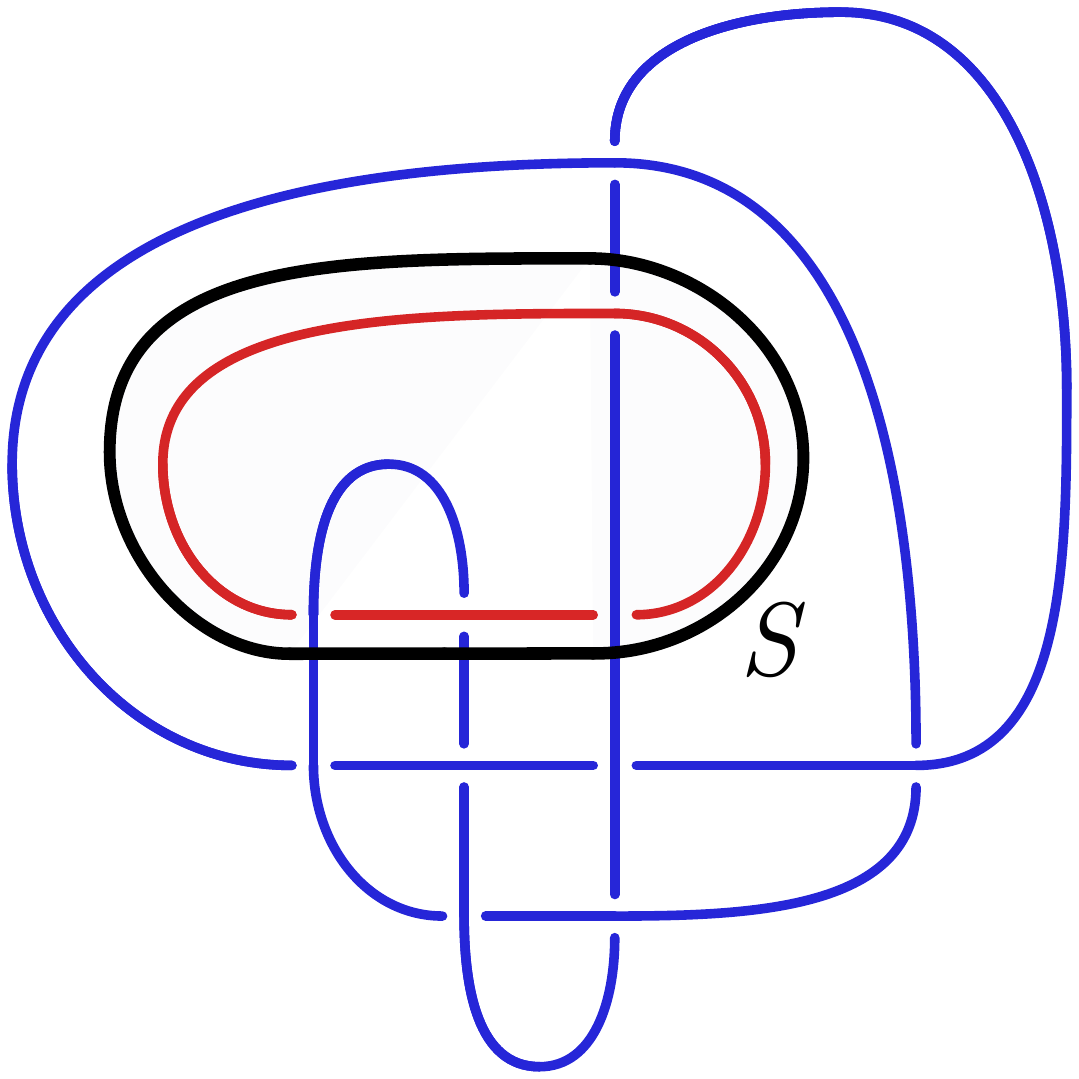}
    \caption{The surface $S$ shown in black is an incompressible $4$-punctured sphere.}
    \label{4PuncSphere}
\end{figure}

Note that $(1,n)$-Dehn filling on $J$ compresses the surface $F$ described above, since the result of $(1,n)$-Dehn filling on $J$ produces a rational tangle on the filled side of $S$.
Although we cannot prove compressibility of all closed embedded essential surfaces in the complement of $L$ upon the result of $(1,n)$-Dehn filling on $J$, we expect this to be the case, and provide some evidence for this below.

The knot $7_6$ is a $2$-bridge knot, and hence its complement  does not contain a closed embedded essential surface (see for example \cite{GL}). Using SnapPy,
we identified that $(-1,1)$, $(1,1)$, $(-1,2)$ and $(1,2)$ Dehn fillings on the component $J$ produce manifolds homeomorphic to the complements of $9_{43}$, $10_{129}$, $K11n57$ and $K12n238$ respectively, all of which are again manifolds
that do not contain a closed embedded essential surface (as can be checked using \cite{BCT} or KnotInfo \cite{knot}). Since the Dehn fillings described above do not contain a closed embedded essential surface, any closed embedded essential surface contained in $S^3\setminus L$ must compress in these Dehn fillings. It follows from
\cite{CGLS} and \cite{Wu} that any closed embedded essential surface in $S^3\setminus L$ must contain an essential simple closed curve that is isotopic to the longitude of $J$. 

From the above discussion, the knots $7_6$, $9_{43}$, $10_{129}$, $K11n57$ and $K12n238$ all have integral trace. 
In particular, $t=\chi_\rho(m_0)$ as in \S \ref{character}, is an algebraic integer, as is the
the solution for $Y$ obtained from $R(t,Y)=0$ in these cases. Hence, at these values of $t$, the polynomial $R(t,Y)$ must be reducible.
We expect this to be the case more generally for $(1,n)$ Dehn filling on $J$.

\subsection{Another link}
\label{L11n71}
Another $2$ component link with an unknotted component, with linking number $2$ between the two components and 
has non-integral trace is the link {\tt{L11n71}} from Thistlethwaite's table \cite{TT}.  Executing the same plan
as we described above leads to analogous polynomial $R_1(t,Y)$ shown below:

\medskip

\begin{align*}
R_1(t,Y) &= 32768 t^5+t^3 Y^{18}-393216 t^3+\big(3 t^5-72 t^3+210 t\big) Y^{16}+\big(-22 t^5+792t^3-2734 t\big) Y^{14} \\
&\quad +\big(104 t^5-5928 t^3+25240 t\big) Y^{12} +\big(-320t^5+30240 t^3-151360 t\big) Y^{10}+\big(256 t^5 \\
&\quad-105728 t^3+660224 t\big) Y^8 +\big(3072 t^5+195584 t^3-1746944 t\big) Y^6+\big(-4096 t^5-192512 t^3 \\
&\quad +3059712t\big) Y^4+\big(-16384 t^5+245760 t^3-3776512 t\big) Y^2+\big(3 t^4-25t^2\big) Y^{17}+\big(456 t^4 \\
&\quad -5968 t^2+4888\big) Y^{13}+\big(-2448 t^4+41104t^2-40752\big) Y^{11}+\big(9024 t^4-201664 t^2 \\
&\quad +223936\big) Y^9+\big(-15872t^4+649216 t^2-881408\big) Y^7+\big(-10240 t^4-1182720 t^2 \\
&\quad+1587200\big) Y^5+\big(73728 t^4+946176 t^2-479232\big) Y^3+\big(-65536 t^4-163840t^2 \\
&\quad-1638400\big) Y+\big(t^6-69 t^4+650 t^2-593\big) Y^{15}+3276800 t.
\end{align*}

\medskip

\noindent with leading term $t^3 Y^{18}$. When $t=2$ (i.e. at the faithful discrete representation) this factors as
$$\left(Y^3+2 Y^2-4 Y-16\right)^2 \left(Y^4-2 Y^3-4 Y^2+8 Y+16\right)^2 \left(8 Y^4-52
   Y^3+132 Y^2-153 Y+68\right)$$
\noindent with the term  $\left(8 Y^4-52Y^3+132 Y^2-153 Y+68\right)$ corresponding to the faithful discrete representation.  

As noted in \S \ref{irred}, for $d$ odd, $2\cos2\pi/d$ is always a unit. Thus specializing the polynomial $R_1(t,Y)$ at such $t=2\cos2\pi/d$ shows that $Y$ is an algebraic integer for all odd $d\geq 2$. 

The knotted component of {\tt{L11n71}} is the knot $7_4$ which is a $2$-bridge knot.  Repeating  the analysis that we did on $L$, we identified that $(-1,1)$, $(1,1)$, $(-1,2)$ and $(1,2)$ Dehn fillings on the 
unknotted component produce manifolds homeomorphic to the complements of $7_{3}$, $10_{130}$, $10_{128}$ and $K12n723$ respectively, which are again all manifolds
that do not contain a closed embedded essential surface (as can be checked using \cite{BCT} or KnotInfo \cite{knot}). Hence these knots have integral trace. Moreover, as with $L$, 
any closed embedded essential surface contained in the complement of {\tt{L11n71}} must compress in these fillings, and as before it follows from
\cite{CGLS} and \cite{Wu}  that any closed embedded essential surface in the complement of {\tt{L11n71}} must contain an essential simple closed curve that is isotopic to the longitude of the unknotted component. 

As with $L$, from the link diagram shown in Thistlethwaite's table \cite{TT}, one sees an essential tangle decomposition of {\tt{L11n71}}, which can be tubed to construct a closed embedded essential surface in the complement of the link {\tt{L11n71}}.

\subsection{The manifold \texorpdfstring{$m137$}{m137}}
\label{m137}
The manifold $m137$ (denoted $M$ in what follows) of the SnapPy census is a knot complement in $S^2\times S^1$, and has been of some interest (see \cite{Dun} and \cite{Gao}).  Moreover, it is the "smallest"
cusped hyperbolic 3-manifold we know of that has non-integral trace.  From SnapPy, a presentation of $\pi_1(M)$ is {\tt{<a,b |aaabbABBBAbb=1>}}, with the faithful discrete representation being given by:
$$a\mapsto \left(
\begin{array}{cc}
 -\frac{3}{2}+\frac{i}{2} & 1 \\
 -1 & 0 \\
\end{array}
\right)~\hbox{and}~b\mapsto \left(
\begin{array}{cc}
 0 & 1 \\
 -1 & -\frac{1}{2}-\frac{i}{2} \\
\end{array}
\right).$$
A peripheral system for $M$ is given by $\{a^{-1}b^2a^4b^2,(ba)^{-1}\}$. Note that $(0,1)$ Dehn filling gives $S^2\times S^1$.  Following \cite{Gao}, set $\lambda=(ba)^{-1}$, then $\pi_1(M)$ can be generated by $\{b,\lambda\}$ and using this, a description for the canonical
component of $M$ is given in \cite{Gao} as the curve in $\C^2$ obtained as the vanishing set of the polynomial: 
$$P(s,t)=(-2-3s+s^3)t^4 + (4+4s-s^2-s^3)t^2-1,$$
where $s=\chi_\rho(\lambda)$, $t=\chi_\rho(b)$ and $\chi_\rho(b\lambda)=t-\frac{1}{t(s+1)}$. 
Note that $(-2-3s+s^3)=(s+1)^2(s-2)$ and $(4+4s-s^2-s^3)=(s+1)(s+2)(s-2)$. Thus, understanding the behavior of $t=\chi_\rho(b)$ (i.e. integral versus non-integral) is reduced to 
understanding when $(s+1)$ and $(s-2)$ are units in the number fields arising from Dehn filling representations.

For example, if we consider $(0,d)$ Dehn fillings with $d$ odd, we are led to consideration of when $(2\cos(2\pi/d)+1)$ and $(2\cos(2\pi/d)-2)$ are and are not units.  For $d$ even similar statements hold
for $(2\cos(2\pi/2d)+1)$ and $(2\cos(2\pi/2d)-2)$. For ease of exposition we will assume that $d$ is odd. 

Now  $(2\cos(2\pi/d)-2)$ is never a unit for $d$ a power of a prime (resp. is a unit when $d$ is not a power of a prime). 
To see this note that: $(2\cos(2\pi/d)-2) = \zeta_d+1/\zeta_d-2 = (\zeta_d-1)^2/\zeta_d$, where $\zeta_d$ is a primitive $d$th root of unity. 
As above, let $\Phi_d(x)$ denote the $d$-th cyclotomic polynomial, then $\zeta_d-1$ is a unit if and only if $\Phi_d(1)=\pm 1$ (see for example \cite[Lemma 2.5]{Len}).
It is a well-known property of cyclotomic polynomials that this happens if and only $d$ is not a power of a prime.  
Similarly, when $(2\cos(2\pi/d)+1)$ is a unit reduces to understanding when $\zeta_d^2+\zeta_d+1$ is a unit, which by \cite[Lemma 2.5]{Len} holds if and only if $\Phi_d(\omega)$ is a unit
where $\omega$ is a primitive cube root of unity.  

We have not analyzed all of this carefully, but experiments seem to support that $(0,d)$ Dehn fillings have integral trace (so modulo irreducibility concerns
both the above terms are units) when $d=10k$, $k\geq 1$. We also found that $(0,14)$ has integral trace. 

Experiments also suggest that many other Dehn fillings have integral traces; for example it seems that for $n$ an integer, the family of $(1,n)$ Dehn fillings have integral trace. In particular, we checked this holds for integers $n\in[-7,-3]\cup[2,6]$ and so
at such Dehn fillings $s$ and $t$ will be algebraic integers.
Hence in these cases, from the expression for $P(s,t)$ (modulo irreducibility concerns), we deduce that $(s+1)$ and $(s-2)$ must be units in the number fields constructed by these Dehn fillings.

We also note that using \cite{CGLS} and \cite{Wu} any closed embedded essential surface in $M$ must contain an essential simple closed curve that is isotopic to $\lambda$.  To see this, as noted above, $(0,1)$ Dehn filling produces $S^2\times S^1$, any closed embedded essential surface in $M$ must compress in this filling. Moreover, 
SnapPy shows that $(-1,3)$, $(-1,4)$, $(-1,5)$, $(1,2)$, $(1,3)$ and $(1,5)$ are all hyperbolic, all have volume $<3$ and have a shortest closed geodesic
of length $>0.3$. Hence using the list of small volume Haken manifolds from \cite{Dun2} all of these manifolds are non-Haken hyperbolic 3-manifolds.
Hence any closed embedded essential surface in $M$ must compress in these fillings. 
In addition, $(-1,1)$ Dehn filling results in a small Seifert fibered space, and so again, any closed embedded essential surface in $M$ must compress in this filling. 

From \cite{CGLS} and \cite{Wu}, the only way that all of these compressions can happen is that any closed embedded essential surface in $M$ must contain an essential simple closed curve that is 
isotopic to $\lambda$.

\section{Additional Mathematica output}
\label{additional}
\begin{align*}
Q(t,X,Y) &=-t+X^{16} Y^3+16 X^{15} Y^4-6 X^{15} Y^2+120 X^{14} Y^5-83 X^{14} Y^3+12 X^{14} Y +560 X^{13}Y^6\\
&\quad  -532 X^{13} Y^4 + 132 X^{13} Y^2-8 X^{13}+1820 X^{12} Y^7-2093 X^{12} Y^5+644X^{12}Y^3-44 X^{12} Y \\
&\quad+4368 X^{11} Y^8-5642 X^{11} Y^6+1800 X^{11} Y^4 +10 X^{11} Y^2 -32 X^{11}+8008 X^{10} Y^9\\
&\quad -11011 X^{10} Y^7+3036 X^{10}Y^5+755X^{10}Y^3-236 X^{10} Y+11440 X^9 Y^{10}-16016 X^9 Y^8 \\
&\quad+2684 X^9 Y^6+3040 X^9 Y^4-700 X^9 Y^2-48X^9+12870 X^8 Y^{11}-17589 X^8 Y^9 \\
&\quad-396 X^8 Y^7+6519 X^8 Y^5 -939 X^8 Y^3-328 X^8Y+11440 X^7 Y^{12}-14586 X^7 Y^{10} \\
&\quad-4752 X^7 Y^8+8892 X^7 Y^6-72 X^7Y^4  -950 X^7
   Y^2-32 X^7+8008 X^6 Y^{13}-9009 X^6 Y^{11} \\
&\quad -7260 X^6 Y^9+8070 X^6 Y^7+1764 X^6 Y^5 -1493 X^6 Y^3-182 X^6 Y+4368 X^5 Y^{14} \\
&\quad -4004 X^5 Y^{12}-6468 X^5 Y^{10}+4800 X^5 Y^8+3024 X^5 Y^6 -1328 X^5 Y^4 -428 X^5 Y^2\\
&\quad-12 X^5+1820 X^4 Y^{15}-1183 X^4 Y^{13}-3828 X^4Y^{11}+1690 X^4Y^9  +2670 X^4 Y^7
    \\
&\quad-583 X^4 Y^5-532 X^4 Y^3-51 X^4 Y +560 X^3 Y^{16}-182 X^3 Y^{14}-1528 X^3 Y^{12}   \\
&\quad+202 X^3 Y^{10}+1416 X^3 Y^8+ 2X^3 Y^6-368 X^3 Y^4-82 X^3Y^2 -6 X^3+120 X^2 Y^{17}  \\
&\quad+7 X^2 Y^{15}-396 X^2 Y^{13}-89 X^2 Y^{11}+448 X^2 Y^9+125 X^2Y^7-134 X^2 Y^5-61 X^2 Y^3\\
&\quad-15 X^2 Y+16 X Y^{18}+8 X Y^{16}-60 X Y^{14}-40 X Y^{12}+76 X Y^{10}+52 X Y^8-20 X Y^6 \\
&\quad-20 X Y^4-12 X Y^2 +Y^{19}+Y^{17}-4 Y^{15}-5 Y^{13}+5Y^{11}+7 Y^9-2 Y^5-3 Y^3+Y.
\end{align*}
\begin{align*}
R(t,Y)&=669124 t-2 t^7-498002 t^5-5223073 t^3-16 t Y^{24}+\left(120 t^2+176\right) Y^{23}+\big(-t^5-344 t^3 \\
&\quad-1595t\big) Y^{22}+\left(-t^7-265 t^5-8323 t^3-5017 t\right) Y^{20}+\big(31 t^7-820 t^5+45501 t^3\\ 
&\quad+26034 t\big) Y^{18}+ \left(-428 t^7+34065 t^5-60100 t^3-223825 t\right) Y^{16}+\big(3393 t^7-229701 t^5\\
&\quad-1671221 t^3-1389221 t\big) Y^{14}+\left(-16709 t^7+392665 t^5+4196073 t^3+3978713 t\right) Y^{12}\\
&\quad+\left(51769 t^7+613384 t^5+1570051 t^3+257774 t\right) Y^{10}+\big(-97592 t^7-3180386 t^5\\
&\quad-27592720 t^3-28733690 t\big) Y^8+\left(102474 t^7+3256419 t^5+42551766 t^3+53431661 t\right) Y^6\\
&\quad+\left(-49677 t^7+1658479 t^5-6346815 t^3-21240713 t\right) Y^4+\big(6945 t^7-5819870 t^5\\
&\quad-50037327 t^3-50675755 t\big) Y^2+ \left(2 t^6+466 t^4+5400 t^2+1265\right) Y^{21}+\big(8 t^6+5340 t^4\\
&\quad-4891 t^2-551\big) Y^{19}+\left(246 t^6-65918 t^4-71499 t^2+10156\right)Y^{17}+ \big(-8510 t^6+292550 t^4\\
&\quad+1114568 t^2+263159\big) Y^{15} +\left(62972 t^6+480016 t^4+532043 t^2-387\right) Y^{13} +\big(-184968 t^6\\
&\quad-4075296 t^4-11015955 t^2-1827985\big) Y^{11}+\big(148666 t^6+7363350 t^4+27163139 t^2 \\
&\quad +4743016\big) Y^9+\left(389244 t^6+2024132 t^4-5822600 t^2+2654010\right) Y^7+\big(-959338 t^6\\
&\quad-19599946 t^4-57150066 t^2-22718115\big) Y^5
   +\big(659693 t^6+19149660
   t^4+77616992 t^2\\
&\quad+31164769\big) Y^3+\left(t^8+235110 t^6+11747029 t^4+26741431 t^2-669124\right) Y
\end{align*}
\medskip

\noindent As a check, Mathematica shows that $R(-2,(17+3\sqrt{-7})/8)=0$ (i.e. at the faithful discrete representation).

\medskip

\begin{align*}
S(X,Y) &= \left(-16 X^9-16 X^7\right) Y^{24}+\left(120 X^{10}+416 X^8+120 X^6\right) Y^{23} \\
&\quad +\left(-X^{13}-349 X^{11}-2637 X^9-2637 X^7-349 X^5-X^3\right) Y^{22} \\
&\quad +\left(2 X^{14}+478 X^{12}+7294 X^{10}+14901 X^8+7294 X^6+478 X^4+2 X^2\right)Y^{21}\\
&\quad +\left(-X^{15}-272 X^{13}-9669 X^{11}-32671 X^9-32671 X^7-9669 X^5-272 X^3-X\right) Y^{20} \\
&\quad +\left(8 X^{14}+5388 X^{12}+16589 X^{10}+21867 X^8+16589 X^6+5388 X^4+8 X^2\right) Y^{19} \\
&\quad+\big(31 X^{15}-603 X^{13}+42052 X^{11}+155422 X^9+155422 X^7+42052 X^5 \\
    &\qquad -603 X^3+31 X\big) Y^{18}\\
&\quad+\left(246 X^{14}-64442 X^{12}-331481 X^{10}-523430 X^8-331481 X^6-64442 X^4+246 X^2\right) Y^{17}\\
&\quad+\big(-428 X^{15}+31069 X^{13}+101237 X^{11}-78455 X^9-78455 X^7+101237 X^5\\
    &\qquad +31069 X^3-428 X\big) Y^{16} \\
&\quad+\big(-8510 X^{14}+241490 X^{12}+2157118 X^{10}+4077395 X^8+2157118 X^6 \\
    &\qquad +241490 X^4-8510 X^2\big) Y^{15}\\
&\quad+\big(3393 X^{15}-205950 X^{13}-2748473 X^{11}-8581139 X^9-8581139 X^7-2748473 X^5\\
    &\qquad-205950 X^3+3393 X\big) Y^{14} \\
&\quad+\big(62972 X^{14}+857848 X^{12}+3396687 X^{10}+5203235 X^8+3396687 X^6+857848 X^4\\
    &\qquad+62972 X^2\big) Y^{13} \\
&\quad+\big(-16709 X^{15}+275702 X^{13}+5808509 X^{11}+19908767 X^9+19908767 X^7 \\
    &\qquad+5808509 X^5+275702 X^3-16709 X\big) Y^{12}\\
&\quad+\big(-184968 X^{14}-5185104 X^{12}-30091659 X^{10}-52011031 X^8-30091659 X^6 \\
    &\qquad-5185104 X^4-184968 X^2\big) Y^{11} \\
&\quad+\big(51769 X^{15}+975767 X^{13}+5724120 X^{11}+12913682 X^9 +12913682 X^7\\
    &\qquad+5724120 X^5+975767 X^3+51769 X\big) Y^{10} \\
&\quad+\big(148666 X^{14}+8255346 X^{12}+58846529 X^{10}+106222714 X^8+58846529 X^6\\
    &\qquad+8255346 X^4+148666 X^2\big) Y^9\\
&\quad+\big(-97592 X^{15}-3863530 X^{13}-45544082 X^{11}-146731430 X^9\\
&\qquad-146731430 X^7-45544082 X^5-3863530 X^3-97592 X\big) Y^8\\
&\quad+\big(389244 X^{14}+4359596 X^{12}+8112588 X^{10}+10938482 X^8+8112588 X^6+4359596 X^4\\
    &\qquad+389244 X^2\big) Y^7\\
&\quad+\big(102474 X^{15}+3973737 X^{13}+60985815 X^{11}+217237739 X^9+217237739 X^7\\
    &\qquad+60985815 X^5+3973737 X^3+102474 X\big) Y^6\\
&\quad+\big(-959338 X^{14}-25355974 X^{12}-149939920 X^{10}-273804683 X^8-149939920 X^6\\
    &\qquad-25355974 X^4-959338 X^2\big) Y^5\\
&\quad+\big(-49677 X^{15}+1310740 X^{13}+902363 X^{11}-25435063 X^9-25435063 X^7+902363 X^5\\
    &\qquad+1310740 X^3-49677 X\big) Y^4\\
&\quad+\big(659693 X^{14}+23107818 X^{12}+164111027 X^{10}+314490573 X^8+164111027 X^6\\
    &\qquad+23107818 X^4+659693 X^2\big) Y^3\\
&\quad+\big(6945 X^{15}-5771255 X^{13}-78990832 X^{11}-258743361 X^9-258743361 X^7\\
&\qquad-78990832 X^5-5771255 X^3+6945 X\big) Y^2\\
&\quad+\big(X^{16}+235118 X^{14}+13157717 X^{12}+77256253 X^{10}+127998182 X^8+77256253 X^6\\
&\qquad+13157717 X^4+235118 X^2+1\big) Y\\
&\quad-2 X^{15}-498016 X^{13}-7713125 X^{11}-19980185 X^9-19980185 X^7-7713125 X^5\\
&\qquad-498016 X^3-2 X
\end{align*}

\section{Knots through \texorpdfstring{$12$}{12} crossings with non-integral trace}
\label{knot12}

In this section we list those knots through $12$ crossings that we are able to confirm have non-integral trace.  As noted previously, a knot $K$ with non-integral trace contains a closed embedded essential surface in its complement.  In \cite{BCT},
they show that of the $2977$ knots in the census of non-trivial prime knots with $\leq12$ crossings, $1019$ of these knots contain a closed embedded essential surface in their complement, and it is this list of $1019$ that is our starting point.

We were able to determine whether or not traces were integral or not for $450$ of them, and of those $450$ knots, we determined that $170$ of them have non-integral trace. 
The tables were compiled using recent additions to SnapPy that, in principle, allow one to compute exactly elements of $\Gamma$ whose traces generate the trace-field $\Q(\tr\Gamma)$. Indeed, as is well-known (see \cite[Chapter 3.5]{MR} for example), the trace of every element in $\Gamma$ is an integer polynomial in the traces of any finite generating set of $\Gamma$ together with a finite number of products of the generators, and so these traces suffice to certify non-integral trace in the sense described below. 

We capped the number of digits that the algebraic numbers were computed to as well as their degree (at $50$) to allow for reasonable runtime. 
For those knots that we were unable to decide integral or non-integral,  one needs additional precision or to raise the degree. 

In the tables that follow we list the $170$ knots with non-integral trace, together with rational primes $p$ that {\em certify} non-integrality.  By this we mean that if $S^3\setminus K=\mathbb{H}^3/\Gamma$ then there exists $\alpha \in\Gamma$ with 
$\tr(\alpha)=r/s\in \mathbb{Q}(\tr\Gamma)$ and a prime ideal $\mathcal{P} \subset \mathbb{Q}(\tr\Gamma)$ with $\mathcal{P}|<s>$ of norm $p^a$ for some integer $a>0$. 
In this notation, for the knots constructed in the proof of Theorem \ref{main}, non-integrality was certified by $p=2$.\\[\baselineskip]
\noindent{\bf $129$ knots with $p=2$:}

\begin{center}
    \begin{tabular}{c c c c c c c c c}
9\textsubscript{29} & 9\textsubscript{38} & 10\textsubscript{96} & 10\textsubscript{97} & 10\textsubscript{99} & 11a38 & 11a102 & 11a123 & 11a124\\
11a126 & 11a173 & 11a232 & 11a244 & 11a291 & 11a292 & 11a293 & 11a294 & 11a346\\
11a347 & 11a353 & 11a354 & 11n65 & 11n66 & 11n68 & 11n69 & 11n97 & 11n99\\
11n156 & 12a66 & 12a74 & 12a100 & 12a150 & 12a156 & 12a163 & 12a199 & 12a207\\
12a231 & 12a244 & 12a245 & 12a260 & 12a311 & 12a331 & 12a396 & 12a414 & 12a435\\
12a491 & 12a493 & 12a494 & 12a634 & 12a647 & 12a702 & 12a706 & 12a708 & 12a771\\
12a798 & 12a818 & 12a845 & 12a847 & 12a853 & 12a862 & 12a873 & 12a886 & 12a939\\
12a940 & 12a1059 & 12a1062 & 12a1097 & 12a1124 & 12a1156 & 12a1173 & 12a1261 & 12a1266\\
12a1270 & 12a1288 & 12n49 & 12n50 & 12n51 & 12n52 & 12n53 & 12n100 & 12n101\\
12n102 & 12n140 & 12n141 & 12n156 & 12n158 & 12n175 & 12n176 & 12n201 & 12n202\\
12n203 & 12n204 & 12n211 & 12n245 & 12n246 & 12n247 & 12n253 & 12n254 & 12n257\\
12n258 & 12n259 & 12n265 & 12n266 & 12n267 & 12n268 & 12n269 & 12n270 & 12n329\\
12n330 & 12n331 & 12n364 & 12n365 & 12n423 & 12n484 & 12n494 & 12n495 & 12n496\\
12n518 & 12n600 & 12n601 & 12n602 & 12n605 & 12n665 & 12n672 & 12n690 & 12n694\\
12n695 & 12n697 & 12n888 & & & & & & 
    \end{tabular}
\end{center}

\noindent{\bf $24$ knots with $p=3$:}

\begin{center}
    \begin{tabular}{c c c c c c c c c}
10\textsubscript{90} & 10\textsubscript{93} & 10\textsubscript{122} & 11a288 & 12a389 & 12a430 & 12a868 & 12a1043 & 12a1105\\
12a1109 & 12a1246 & 12n193 & 12n194 & 12n195 & 12n196 & 12n215 & 12n216 & 12n217\\
12n454 & 12n456 & 12n689 & 12n840 & 12n879 & 12n886 &  &
    \end{tabular}
\end{center}


\noindent{\bf $17$ remaining cases:}

\begin{center}
\begin{tabular}{c|c c@{\hspace{0.5in}}c|c c@{\hspace{0.5in}} c|c}
    knot & primes & & knot & primes & & knot & primes \\
    \cline{1-2} \cline{4-5} \cline{7-8}
    10\textsubscript{98} &  2,3 & & 12a567 & 23 & & 12n264 & 7 \\
    11a132 &  2,3 & & 12a701 & 2,5 & & 12n440 & 2,3 \\
    11a323 & 5 & & 12a1117 & 13 & & 12n508 & 2,3 \\
    12a348 & 2,3 & & 12a1203 & 7 & & 12n604 & 2,3 \\
    12a466 & 7 & & 12a1205 & 17 & & 12n868 & 5 \\
    12a483 & 7 & & 12n256 & 7 & &  &  \\
\end{tabular}
\end{center}

\section{Questions and comments}
\label{questions}

We gather together some questions raised by this work, as well as some comments.\\[\baselineskip]
\noindent{\bf Existence of accidental parabolic elements:}~In the two examples of link complements considered in this paper, as well as the example of $m137$, the closed embedded 
essential surfaces in the these manifolds carried essential curves that were isotopic to essential simple curves on a boundary torus, these are examples of {\em accidental parabolic elements} in the surface group.
As we now describe, this also holds for all of the knots listed in \S \ref{knot12}.

For the alternating knots listed in \S \ref{knot12}, this follows from \cite{Men}, for which the accidental parabolic is a meridian. As
was pointed out to us by J. Howie, all the non-alternating knots listed in \S \ref{knot12} (apart from {\tt{12n253}} and {\tt{12n254}}) are almost alternating, and so by \cite{Aetal} also have
complements for which the meridian is an accidental parabolic on any closed embedded essential surface. Furthermore, Howie observed that the two remaining knots admit an essential tangle decomposition as shown below in
Figure \ref{tangleDecomp}. Tubing the essential $4$-punctured spheres $S$ and $S'$ shown in Figure \ref{tangleDecomp} provides a closed embedded essential surface that carries an accidental parabolic which is a meridian.

\begin{figure}[htpb]
    \centering
    \includegraphics[scale=0.5]{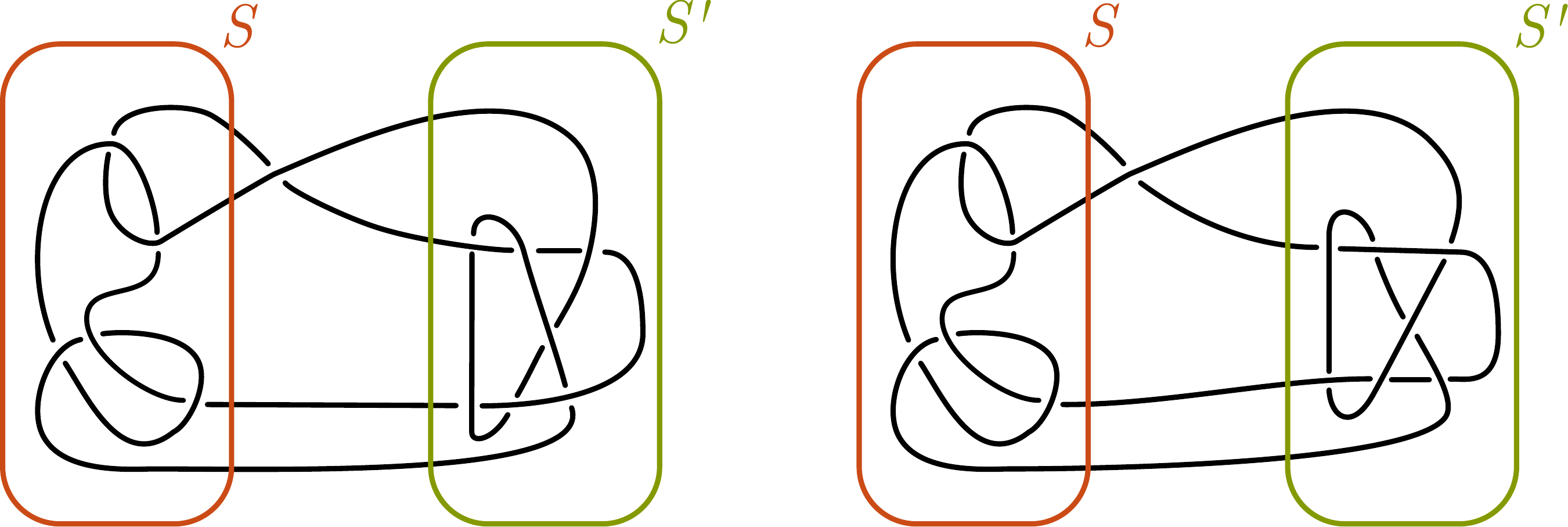}
    \caption{The knot \texttt{12n253} on the left and \texttt{12n254} on the right. The orange and green essential $4$-punctured spheres, $S$ and $S'$, provide an essential tangle decomposition for each knot.}
    \label{tangleDecomp}
\end{figure}

Howie also pointed out to us that the knots constructed in the proof of Theorem \ref{main} have complements that admit a closed embedded essential surface that carries an accidental parabolic element (again a meridian).  We include his argument below.

\begin{lemma}[Howie]
\label{howie}
Let $K_d$ denote the knot constructed in the proof of Theorem \ref{main} via the $d$-fold cyclic branched cover of $S^3$ branched over $J$, and where $d$ is assumed to be odd. Then $S^3\setminus K_d$ contains a closed
embedded essential surface for which the meridian is an accidental parabolic.\end{lemma}

\begin{proof} Performing an isotopy to the link $L$ results in the diagram shown in Figure \ref{Surface}.

\begin{figure}[htpb]
    \centering
    \includegraphics[scale=0.8]{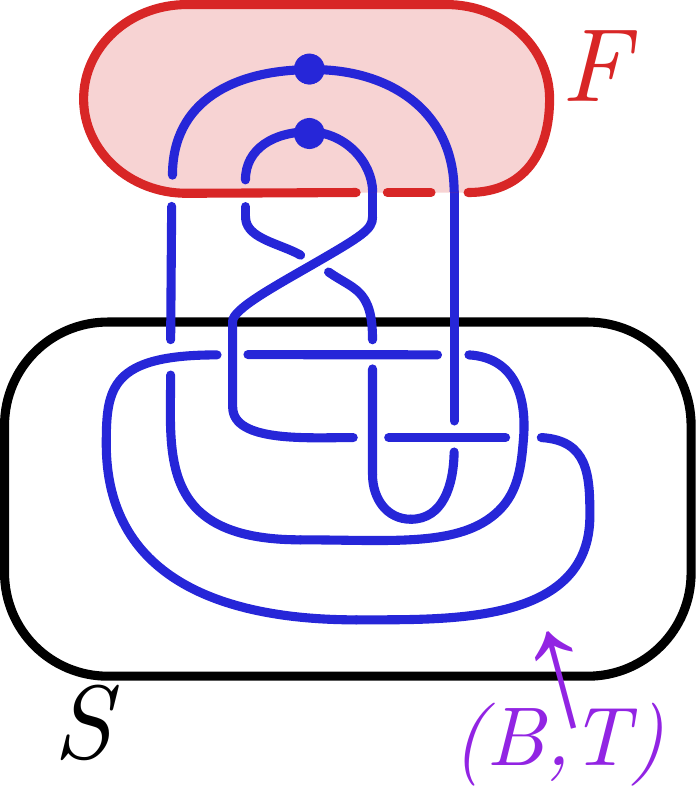}
    \caption{The link $L$ after isotopy. The shaded red surface $F$ is a Seifert surface. The black surface $S$ is a $4$-punctued sphere, and $(B,T)$ is the ball-tangle pair.}
    \label{Surface}
\end{figure}

The $d$-fold cyclic branched cover over $J$ that we may made use of in the proof Theorem \ref{main} can be described as follows. Cut along the Seifert surface $F$ shown in Figure \ref{Surface}, cyclically glue $d$ copies of the resulting piece, and then glue a solid torus $J'$ back in (where $J'$ is the lift of $J$). Since the $4$-punctured sphere $S$ is disjoint from $J$ and $F$, $S$ will lift to $d$ disjoint copies of itself, which we denote $S'$, $S''$, etc. Similarly the ball-tangle pair $(B,T)$ lifts to $d$ disjoint copies of itself denoted $(B',T')$, 
$(B'',T'')$, etc (see Figure \ref{lifts}). Note that since $T$ is an essential tangle, $S$ is incompressible to one side of $(B,T)$. \par
\begin{figure}[htpb]
    \centering
    \includegraphics[scale=0.5]{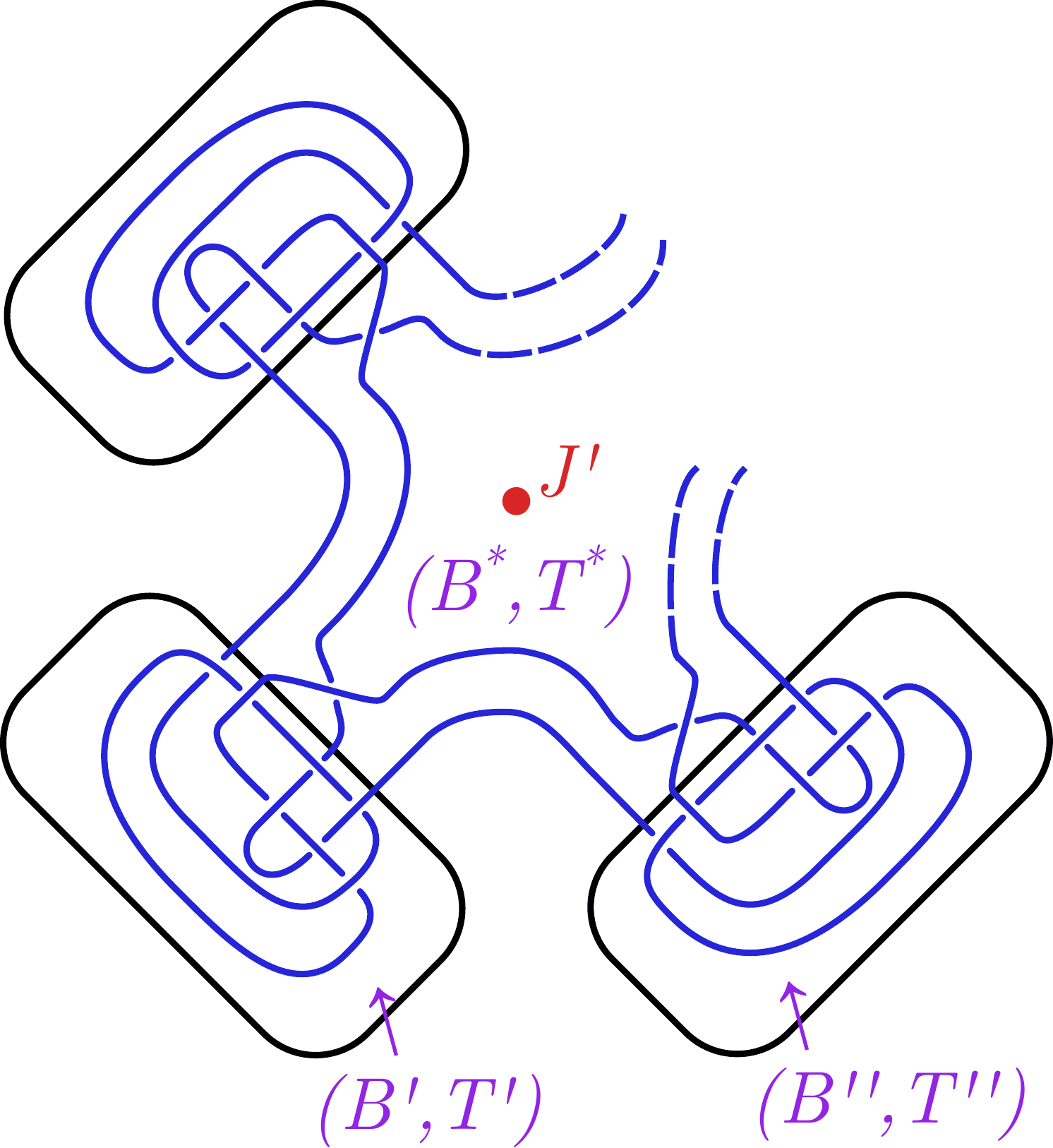}
    \caption{The $d$-fold cyclic branched cover over $J$.}
    \label{lifts}
\end{figure}
Now $d\geq 3$ is odd, so $S'$ is incompressible to one side of $(B',T')$, since $(B',T')$ is simply a lift of the ball-tangle pair $(B,T)$. It remains to show $S'$ is incompressible on the other side. As shown in Figure \ref{lifts}, we write $(B^*,T^*)$
for the ball-tangle pair on the other side of $S'$. Note that $T^*$ contains no closed components. Let $D$ be a compressing disk for $S'$ in $(B^*,T^*)$. Then $D$ must separate the two strands of $T^*$. It follows that
$D$ must intersect $S''$, otherwise it would fail to separate the two strands of $T''$ which belong to different strands of $T^*$.

Now consider the intersection pattern of $S''$ on $D$. Curves which are trivial on $S''$ can be removed, and the only curves which remain belong to a non-empty family of parallel curves which separate pairs of points on $S''$. Choosing an innermost such curve determines a loop in $D$ which bounds a disk $D'$ say. Now $D'$ cannot bound a disk in $(B'',T'')$ since as above, $(B'',T'')$ is simply a lift of the ball-tangle pair $(B,T)$. Moreover, $D'$ cannot bound a disk on the other side of $S''$, since arguing as above,
$D'$ would have to intersect $S'$. However, the interior of $D$ is disjoint from $S'$ (since it is a compressing disk), so $S'$ is incompressible to both sides, and therefore is essential in the complement of $K_d$.

We can then find at least one tubing of $S'$ that produces a closed embedded essential surface in the complement of $K_d$. By construction the meridian is an accidental parabolic.\end{proof}

We also note that there is a knot with non-integral trace for which the meridian cannot be an accidental parabolic.  The knot in question is 
{\tt{15n153789}} which appeared in \cite{DGR} (as an example of a ``barely large knot") and contains a unique closed embedded essential surface $S$ of genus $2$. Now 
\cite[Theorem 7.6]{DGR} shows that the meridian is not a boundary slope, and so $S$ cannot contain an essential simple closed curve isotopic to a meridian.
We do not know whether this surface carries an accidental parabolic. 
That it has non-integral trace can be checked using Snap or SnapPy.

\medskip

\noindent Given this discussion, it seems reasonable to ask:\\[\baselineskip]
\noindent {\bf Question 1:}~{\em Does every knot $K$ with non-integral trace have a complement that contains a closed embedded essential surface containing an accidental parabolic element?}\\[\baselineskip]
\noindent{\bf $2$-generator non-integral knots:}~A $3$-manifold $M$ is called $2$-generator if $\pi_1(M)$ can be generated by two elements.  A link $L\subset S^3$ is called $2$-generator if $\pi_1(S^3\setminus L)$
is $2$-generator.  The two links {\tt{L11n106}} and {\tt{L11n71}} considered in this paper, as well as the example of $m137$ are $2$-generator (which greatly facilitated computation). On the other hand, none of the $170$ examples listed in \S \ref{knot12} appear to be ``obviously" $2$-generator (using SnapPy), and in a previous version of this paper we asked whether 
there exists a hyperbolic knot $K\subset S^3$ with non-integral trace which is $2$-generator.

The following example was pointed out to us by K. Baker and N. Hoffman.  The manifold {\tt{v1980}} of the SnapPy census is homeomorphic to the complement of a Berge knot $K\subset S^3$ which they checked by SnapPy has non-integral trace. Indeed, in the terminology of \cite{Ba}, $K$ is a knot which lies on the fiber of the trefoil knot complement and so the knot arises from Berge's family VII.  

Being a Berge knot, $K$ is $2$-generator, and has a Lens Space Dehn filling. In particular, it is an {L}-space knot in the sense of Ozsv{\'a}th and Szab{\'o} \cite{OS}, and so this example also answers another question from an earlier version of this paper, namely
whether there exists a knot $K$ with non-integral trace that is an {L}-space knot in the sense of Ozsv{\'a}th and Szab{\'o}.\\[\baselineskip]
\noindent{\bf Non-triviality of the Alexander polynomial:}~The manifold $m137$ has trivial Alexander polynomial, however it can be checked from \cite{knot} for example, that none of the $170$ knots in \S \ref{knot12} have trivial Alexander polynomial. Moreover, as we now show, the knots constructed in Theorem \ref{main} also do not have trivial Alexander polynomial.

\begin{proposition}
\label{Alexnot1}
All the knots constructed in the proof of Theorem \ref{main} have non-trivial Alexander polynomial.\end{proposition}

\begin{proof} The Alexander polynomial of the link $L$ used in the proof of Theorem \ref{main} can be computed in SnapPy using:

\begin{verbatim}
link=snappy.Link('L11n106')
link.alexander_polynomial()
\end{verbatim}

\noindent which gives.
$$\Delta_L(u,v) = u \left(v^5-2 v^4-v^3+3 v^2-2 v\right)-2 v^4+3 v^3-v^2-2 v+1,$$
where $u$ is the meridian of the unknotted component $J$.  Note that $v^5-2 v^4-v^3+3 v^2-2 v$ factors as $v(v-2)(v^3-v+1)$

As above, let $K_d$ denote the knot constructed in the proof of Theorem \ref{main} via the $d$-fold cyclic branched cover of $S^3$ branched over $J$, and where $d$ is assumed to be odd.  Using \cite[Proposition 4.1 \& Theorem 1]{Mur} and the fact that the branch locus is the unknot, it follows that the Alexander polynomial of $K_d$ is given by:
$$\Delta_{K_d}(u) = \prod_{i=1}^{d-1} \Delta_L(u,\zeta_d^i),$$
where $\zeta_d$ is a primitive $d$-th root of unity.

Note that this product produces a polynomial in $u$ of degree $d-1$ with leading coefficient $\prod_{i=1}^{d-1}\zeta_d^{i}(\zeta_d^{i}-2)({\zeta_d^{i}}^{3}-\zeta_d^{i}+1)$. Neither of $(\zeta_d^{i}-2)$ or $({\zeta_d^{i}}^{3}-\zeta_d^{i}+1)$ are factors of cyclotomic polynomials and so the product is never zero.  Hence $\Delta_{K_d}(u)\neq 1$ as required.\end{proof}

\noindent {\bf Question 3:}~{\em Does there exist a hyperbolic knot $K$ with non-integral trace and with trivial Alexander polynomial?}



\end{document}